\documentclass[11pt,a4paper, final, twoside]{article}
\usepackage{amsmath}
\usepackage{fancyhdr}
\usepackage{amsthm}
\usepackage{amsfonts}
\usepackage{amssymb}
\usepackage{amscd}
\usepackage{amsthm}
\usepackage{graphicx}
\usepackage{afterpage}
\usepackage{xmpincl}
\usepackage[a-1b]{pdfx} 
\usepackage{inputenc}
\usepackage{hyperref}

\hypersetup{pdfencoding=unicode,
    colorlinks=true,		
	urlcolor=blue,  
		linkcolor=blue, 
		citecolor=blue
}

  \inputencoding{utf8}
  \makeatother
	
\newtheorem{theorem}{Theorem}

\newtheorem{definition}[theorem]{Definition}
\newtheorem{example}[theorem]{Example}

\newtheorem{lemma}[theorem]{Lemma}

\newtheorem{remark}[theorem]{Remark}

\numberwithin{equation}{section}

\begin{document}
\hyphenpenalty=10000

\begin{flushright}
{\footnotesize \textbf{ISSN} 1842-6298 (electronic), 1843-7265 (print)}

{\footnotesize \href{http://www.utgjiu.ro/math/sma/v13/v13.html}{Volume
13~(2018)}, xx -- yy \\[15mm]
}
\end{flushright}

\begin{center}

{\Large \textbf{ON TEA, DONUTS AND NON-COMMUTATIVE GEOMETRY }}\\[5mm]
{\large {Igor ~V. ~Nikolaev  }\\[10mm]
}
\end{center}

\newcommand{\Coh}{\hbox{\bf Coh}}
\newcommand{\Mod}{\hbox{\bf Mod}}
\newcommand{\Tors}{\hbox{\bf Tors}}

{\footnotesize \textbf{Abstract}. 
As many will agree, it feels good to complement a cup of
tea by a donut or two.  This sweet relationship
is also a guiding principle of  non-commutative geometry
known as Serre Theorem. We explain the algebra
behind this theorem and prove that elliptic curves
are complementary to the so-called non-commutative tori.}
\afterpage{
\fancyhead{} \fancyfoot{} 
\fancyhead[LE, RO]{\bf\thepage}
\fancyhead[LO]{\small On tea, donuts and non-commutative geometry}
\fancyhead[RE]{\small Igor V. Nikolaev}
\fancyfoot[C]{\small******************************************************************************\\
Surveys~in~Mathematics~and~its~Applications
\href{http://www.utgjiu.ro/math/sma/v13/v13.html}{{\bf 13}~(2018)}, xx -- yy \\
 \href{http://www.utgjiu.ro/math/sma}{http://www.utgjiu.ro/math/sma}}}

\section{Introduction}

\centerline{{\it\quad ``...in divinity opposites are always reconciled.''}\qquad
 ―--- Walter M. Miller Jr.}

\bigskip\noindent
 An algebraic curve $C$  is the set of points on the affine plane whose coordinates are zeros of a 
 polynomial in two variables with   real or complex coefficients, 
 like the one shown in Figure 1. 
\begin{figure}[h]
\begin{picture}(300,100)(-50,0)

\put(100,50){\line(1,0){60}}
\put(130,20){\line(0,1){60}}

\thicklines
\qbezier(165,25)(115,50)(165,75)
\qbezier(110,50)(117,15)(130,50)
\qbezier(110,50)(117,85)(130,50)

\put(110,55){\line(4,1){40}}
\put(150,65){\line(0,-1){30}}

\put(100,55){$a$}
\put(120,70){$b$}
\put(155,33){$a+b$}

\end{picture}
\caption{Affine cubic  $y^2=x(x-1)(x+1)$
with addition law.}
\end{figure}
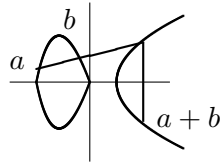

The mathematical theory of algebraic curves emerged from the concept
of an analytic surface created by Georg Friedrich Bernhard Riemann  (1826-1866).
These are called Riemann surfaces and  are nothing but  algebraic
curves  over the field ${\Bbb C}$.  
The  proper  mathematical language --  algebraic geometry -- 
is the result of an inspiration and the hard work of 
Julius Wilhelm Richard Dedekind  (1831-1916), Heinrich Martin Weber (1842-1913),
David Hilbert (1862-1943),  Wolfgang Krull (1899-1971),  Oscar Zariski (1899-1986), 
Alexander Grothendieck (1928-2014) and Jean-Pierre Serre (born 1926)
among others.

The ${\Bbb C}$-valued polynomial functions defined on the curve $C$ can be added and multiplied pointwise. 
This makes the totality of  such functions into a ring, say, $A$. 
Since multiplication of complex numbers is commutative, 
$A$ is a {\it commutative ring}. As an algebraic object, it is \textit{dual} to the geometric object $C$.
In particular, $C$ can recovered up to isomorphism from $A$ and vice versa. In simple terms,
this is the duality between systems of algebraic equations and their solutions.

Instead of the ring $A$ itself,  it is often useful 
 to work with   {\it modules} over $A$;  the latter is a powerful tool of modern 
algebra synonymous   with  representations of   the ring $A$.
Moreover,   according to Maurice Auslander (1926-1994), one should study 
morphisms between modules rather than modules themselves \cite{Aus1}. 
Here  we are talking  about the {\it category of modules},   i.e. the collection of all modules  and all morphisms 
between them.  (In other words,  instead of the individual modules over $A$ we shall study
their ``sociology",  i.e. relationships between modules over $A$.)   
We shall write $\Mod~(A)$ to denote a category of 
 finitely generated graded modules  over $A$ factored by  a category  of
 modules of  finite length.
In  geometric terms,   the r\^ole of  modules  is  played by the so-called {\it sheaves}  on the curve $C$.  
Likewise,   we are looking at the  category  $\Coh~(C)$ of {\it  coherent}  sheaves  on $C$,  
i.e. a class of sheaves having particularly nice  properties related to the geometry of $C$. 
Now  the  duality  between curve   $C$ and algebra  $A$ is  
a special case of  the famous {\it Serre Theorem} \cite{Ser1}
saying that the two categories are equivalent: 
\begin{equation}\label{eq2}
\Coh~(C)\cong \Mod~(A).  
\end{equation}
The careful reader may notice that   $\Mod~(A)$  is well-defined for all
rings -- commutative or {\it not}.  For instance,  consider the non-commutative  
ring $M_2(A)$ of  two-by-two  matrices over  $A$.   It is a trivia to verify  that 
$\Mod~(M_2(A))\cong  \Mod~(A)$ are equivalent categories, yet 
  the ring $M_2(A)$ is not isomorphic to $A$.  
(Notice that $A$ is the center of the ring  $M_2(A)$.)
Rings whose module categories are equivalent are said to be 
{\it Morita equivalent}.

In 1982 Evgeny Konstantinovich Sklyanin (born 1955) was busy with a difficult 
problem of  quantum physics \cite{Skl1},  when he discovered  a remarkable
 non-commutative  ring $S(\alpha,\beta,\gamma)$ 
  with the following  property, see Section 4.  If one calculates the right hand side of (\ref{eq2})
for the $S(\alpha,\beta,\gamma)$, then it will  be equivalent to  the left hand side of (\ref{eq2})  
calculated for an {\it elliptic curve};   by such we  understand a subset ${\cal E}$ of
complex projective plane ${\Bbb C}P^2$ given by the {\it Legendre} cubic: 
\begin{equation}\label{eq3}
 y^2z=x(x-z)(x-\lambda z),\qquad \lambda\in {\Bbb C} - \{0; 1\}. 
\end{equation}
Moreover,  the ring $S(\alpha,\beta,\gamma)$ gives rise to an automorphism  $\sigma: {\cal E}\to {\cal E}$ \cite{Skl1}.
Sklyanin's example hints,  that at least some parts of algebraic geometry can be recast in 
terms of non-commutative algebra;  but what such a generalization is good for?

 Long time ago, Yuri Ivanovich Manin (born 1937)
and his student Andrei Borisovich Verevkin (born 1964) suggested that   in the future  all 
algebraic geometry will be  non-commutative;  the classical case  corresponds to  an  ``abelianization'' 
of  the latter \cite{Ver1}. No strict description of such a process for  rings exists,
but for  groups one can   think of the abelianized fundamental group of a knot or a manifold; 
the latter is still a valuable topological invariant -- the first homology -- yet it cannot 
distinguish between  knots or manifolds nearly as good as the fundamental group does.
This simple observation points to  advantages of non-commutative geometry.

In this note we demonstrate  an equivalence  between the category of elliptic curves and 
the category of  so-called  {\it non-commutative tori};  the latter  is  closely related (but not identical) 
to the  category of Sklyanin  algebras,  see Section 4.     
 We  explore  an application   to the rank problem for elliptic curves,
 see Section 5.

\section{Elliptic curves on breakfast}
An imaginative reader may argue that the affine cubic in Figure 1 reminds
a part of a tea pot;  what will be the donut in this case? 
(The situation is sketched in Figure 2 by  the artistically impaired author.)  
We show in Section 4
that the algebraic dual of ${\cal E}$ (the ``donut'') corresponds to 
a {\it non-commutative torus} to be defined below. But first, let us 
recall some standard facts.  
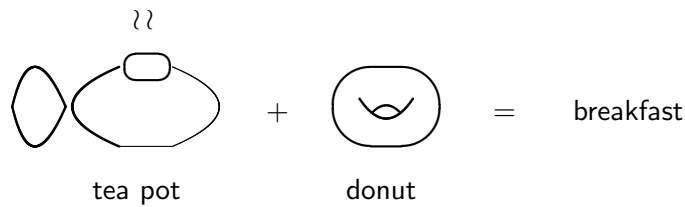
\begin{figure}[h]
\begin{picture}(300,100)(50,0)

\put(150,35){\line(1,0){20}}

\qbezier(170,35)(205,50)(170,65)

\thicklines


\put(250,50){\oval(40,30)}
\qbezier(240,53)(250,38)(260,53)
\qbezier(245,48)(250,53)(255,48)


\qbezier(150,35)(115,50)(150,65)
\qbezier(110,50)(117,20)(130,50)
\qbezier(110,50)(117,80)(130,50)

\put(160.5,65){\oval(17,10)}

\put(205,45){$+$}
\put(290,45){$=$}
\put(320,45){{\sf breakfast}}
\put(140,15){{\sf tea pot}}
\put(235,15){{\sf donut}}

\put(155,80){$\wr$}
\put(160,80){$\wr$}

\end{picture}
\caption{Mathematical breakfast.}
\end{figure}

An {\it elliptic curve}  ${\cal E}$ is a subset  of  
 ${\Bbb C}P^2$   given by equation (\ref{eq3}).  There exist two more equivalent
definitions of ${\cal E}$.

(1)  The curve ${\cal E}$ can be  defined as the  intersection of two quadric surfaces
in   complex projective space  ${\Bbb C}P^3$:
\begin{equation}\label{jacobi}
\left\{
\begin{array}{ccc}
u^2+v^2+w^2+z^2 &=&  0,\\
Av^2+Bw^2+z^2  &=&  0,   
\end{array}
\right.
\end{equation}
where $A$ and $B$ are some complex constants  and  $(u,v,w,z)\in {\Bbb C}P^3$. 
(Unlike the Legendre form (\ref{eq3}), this representation of ${\cal E}$ is {\it not} unique.)
The system of equations (\ref{jacobi}) is called  the {\it Jacobi form} of ${\cal E}$.

(2)  The analytic form of ${\cal E}$  is given by a  {\it complex torus}, i.e.
a surface of genus 1  whose local charts are patched together by a complex analytic map.  
The latter  is simply  the quotient space  ${\Bbb C}/L_{\tau}$, where 
$L_{\tau}:={\Bbb Z}+{\Bbb Z}\tau$ is a lattice and 
$\tau\in {\Bbb H}:=\{z=x+iy\in {\Bbb C} ~|~ y>0\}$ is called
the  complex modulus,  see Figure 3.   
\begin{figure}
\begin{picture}(300,100)(0,0)

\put(90,90){\line(1,0){80}}
\put(80,70){\line(1,0){80}}
\put(60,30){\line(1,0){80}}

\put(85,20){\line(1,2){40}}
\put(65,20){\line(1,2){40}}
\put(105,20){\line(1,2){40}}
\put(125,20){\line(1,2){40}}

\put(180,50){\vector(1,0){40}}

\thicklines
\put(70,50){\line(1,0){80}}
\put(100,20){\line(0,1){80}}

\put(100,50){\vector(1,0){20}}
\put(100,50){\vector(1,2){10}}
\put(110,70){\line(1,0){20}}
\put(120,50){\line(1,2){10}}


\put(290,50){\oval(60,40)}
\qbezier(280,53)(290,38)(300,53)
\qbezier(285,48)(290,53)(295,48)


\put(275,90){${\Bbb C}/L_{\tau}$}
\put(70,90){${\Bbb C}$}
\put(120,35){$1$}
\put(105,75){${\tau}$}

\put(180,55){quotient}
\put(190,40){map}

\end{picture}
\caption{Complex torus  ${\Bbb C}/L_{\tau}$.}
\end{figure}
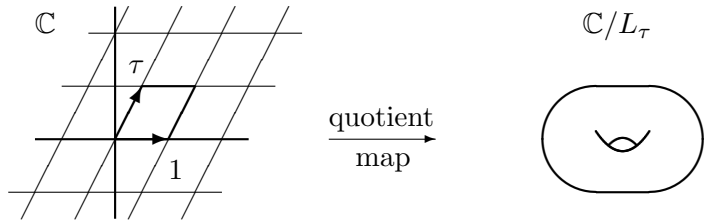

There exists a one-to-one correspondence   ${\cal E}\to {\Bbb C}/L_{\tau}$ 
given  by a meromorphic  $\wp$-function:
\begin{equation}\label{wp} 
\wp(z)={1\over z^2}+\sum_{\omega\in L_{\tau}-\{0\}} \left({1\over (z-\omega)^2}-{1\over\omega^2}\right).
\end{equation}
The lattice $L_{\tau}$ is connected  to the so-called  {\it Weierstrass normal form}  of  ${\cal E}$
in ${\Bbb C}P^2$:
\begin{equation}\label{w}
 y^2z=4x^3-g_2xz^2-g_3z^3,
 \end{equation}
 where $g_2 =  60\sum_{\omega\in  L_{\tau}-\{0\}} {1\over\omega^4}$  and 
$g_3 =  140\sum_{\omega\in L_{\tau}-\{0\}} {1\over\omega^6}$.   
In turn,  the Weierstrass normal form (\ref{w})  is linked to the Legendre cubic  (\ref{eq3})
by the formulas  $g_2={^3\sqrt{4}\over 3}(\lambda^2-\lambda+1)$ and   
$g_3={1\over 27} (\lambda+1)(2\lambda^2-5\lambda+2)$.
Thus,  it is possible to recover   ${\cal E}$  -- up to an isomorphism --  from $L_{\tau}$;  
we shall denote  by ${\cal E}_{\tau}$  the  elliptic curve corresponding  
to the complex torus of modulus $\tau$.  
\begin{remark}
It is easy to see, that each automorphism $\sigma: {\cal E}_{\tau}\to {\cal E}_{\tau}$
is given by a shift of lattice $L_{\tau}$ in the complex plane ${\Bbb C}$ and vice versa;  
thus points $p\in {\Bbb C}/L_{\tau}$
are bijective with the automorphisms of ${\cal E}_{\tau}$.  In particular,  points of finite
order on ${\cal E}_{\tau}$ correspond to  finite order automorphisms of  ${\cal E}_{\tau}$.    
\end{remark}
Notice that $L_{\tau}$ can be written in a new basis $\{\omega_1,\omega_2\}$, 
where $\omega_2=a\tau +b$ and $\omega_1=c\tau+d$ for some $a,b,c,d\in {\Bbb Z}$, 
such that $ad-bc=1$.  Normalization  $\tau:= {\omega_2\over \omega_1}$ of basis   $\{\omega_1,\omega_2\}$ 
to a standard basis $\{1,\tau'\}$ implies   that   ${\cal E}_{\tau}$ and ${\cal E}_{\tau'}$ are isomorphic, 
if and only if:  
\begin{equation}\label{iso}
\tau'={a\tau+b\over c\tau+d} \quad \hbox{for some matrix} 
\quad\left(\begin{matrix} a & b\cr c & d\end{matrix}\right)  \in SL_2({\Bbb Z}).
\end{equation}

An elliptic curve $\mathcal{E}_{\tau}$ is said to have {\it complex multiplication}  (CM) if there is a complex number 
$\alpha \not\in {\Bbb Z}$  such that $\alpha L_{\tau} \subseteq L_{\tau}$. Applying multiplication by $\alpha$ to 
$1 \in L_{\tau}$, we have $\alpha \in L_{\tau}$,  i.e., $\alpha = a + b\tau$ for some $a,b \in {\Bbb Z}$. 
Applying multiplication by $\alpha$ to $\tau \in L_{\tau}$,   we have $(a + b\tau) \tau = c + d\tau$ 
for some $c,d \in {\Bbb Z}$.  Since $\tau$ is not a real number and the ratio $(c +d\tau)(a+b\tau)^{-1}$ 
is uniquely determined, $c + d\tau$ cannot be a real multiple of $a + b\tau$. 
Thus $a + b\tau$ and $c + d\tau$,  viewed as vectors of the real vectors space ${\Bbb C}$, are linearly independent, 
and therefore $ad-bc \neq 0$. Rewriting the above equality as a quadratic equation,
$b\tau^{2} +(a-d)\tau -c =0$, we have  $\tau = \frac{(d-a) \pm \sqrt{-D}}{2b}$,
where $D = -(a-d)^2  -4ac$. Since $\tau$ is not a real number, 
we must have $D > 0$, which tells us that $\tau$ (and hence $\alpha = a + b\tau$) 
belongs to the imaginary quadratic field ${\Bbb Q}(\sqrt{-D})$. 
In particular, not every elliptic curve admits complex multiplication.
Conversely,   if  $\tau\in {\Bbb Q}(\sqrt{-D})$ then $\mathcal{E}_{\tau}$ has CM; 
thus one gets a necessary and sufficient condition for complex multiplication.

No surprise, that  complex multiplication is used to construct   extensions of the imaginary
quadratic  fields with abelian Galois group,  thus  providing  a powerful link between complex analysis
and number theory \cite{ST}.  Perhaps  that is why   David Hilbert counted complex multiplication as not 
only the most beautiful part of mathematics but also of entire science.

\section{A non-commutative donut}
The product of complex numbers $a$ and $b$ is commutative, i.e. $ab=ba$. Such a
property cannot be taken for granted, since more than often $ab\ne ba$.
(For instance, putting  on a sock $a$ then a rollerblade shoe $b$ feels better than the other 
way around.)  Below we consider a non-commutative algebra related to elliptic curves.

A {\it $C^*$-algebra} ${\cal A}$ is an algebra over $\Bbb C$ with a norm 
$a\mapsto ||a||$ and an involution $a\mapsto a^*, a\in {\cal A}$, such that ${\cal A}$ is
complete with  respect to the norm, and such that $||ab||\le ||a||~||b||$ and $||a^*a||=||a||^2$ for every  $a,b\in {\cal A}$.  
We assume that our  algebra ${\cal A}$ has a unit.
Such algebras provide an axiomatic way to describe 
 rings of bounded linear operators acting on a Hilbert space. 
These  rings  appear in the works of John von Neumann (1903-1957) and his former student 
Francis Joseph Murray (1911-1996) on quantum physics. Axioms for $C^*$-algebras
were introduced  by Israel Moiseevich Gelfand (1913-2009). 
The usage of $C^*$-algebras in non-commutative geometry was 
pioneered by Alain Connes (born 1947).

A $C^*$-algebra is called {\it universal} if for any other $C^*$-algebra
 with the same number of generators satisfying the
 same relations, there is a unique unital $C^*$-morphism
 from the former to the latter.  (For $C^*$-algebras, the universal problem does not always
 have a solution, but it does have a solution for a class of  $C^*$-algebras which we introduce below.)
An element  $u\in {\cal A}$  is called  a {\it unitary}, if 
$u^{-1}=u^*$. 
\begin{definition}\label{dfn1}
{\bf (\cite{Rie1})}
By a non-commutative torus one understands the universal $C^*$-algebra 
${\cal A}_{\theta}$ generated by unitaries $u$ and $v$ satisfying 
the commutation relation $vu=e^{2\pi i\theta}uv$ for a real constant $\theta$. 
\end{definition}
\begin{remark}
If one denotes $x_1=u, x_2=u^*, x_3=v, x_4=v^*$ and if $e$ is  the unit, then the 
relations for  ${\cal A}_{\theta}$ take the form: 
\begin{equation}\label{easy}
\left\{
\begin{array}{cc}
x_3x_1  &= e^{2\pi i\theta}x_1x_3,\\
x_1x_2 &= e= x_2x_1,\\
x_3x_4  &= e= x_4x_3,
\end{array}
\right.
\end{equation}
with  the involution  $x_1^*=x_2$ and $x_3^*=x_4$. 
\end{remark}

Different ${\cal A}_{\theta}$ may have equivalent module categories. It was 
shown  by Marc Aristide Rieffel (born 1937) and others, that 
 ${\cal A}_{\theta}$ and ${\cal A}_{\theta'}$ are Morita equivalent
 if and only if: 
\begin{equation}\label{mor}
\theta'={a\theta+b\over c\theta+d} \quad \hbox{for some matrix} 
\quad\left(\begin{matrix} a & b\cr c & d\end{matrix}\right)  \in SL_2({\Bbb Z}).
\end{equation}

Denote by $\Lambda_{\theta}$ a {\it pseudo-lattice},  i.e. the  abelian subgroup ${\Bbb Z}+{\Bbb Z}\theta$ 
of ${\Bbb R}$.    The ${\cal A}_{\theta}$ is said to have {\it real  multiplication} (RM), if there 
exists a real number $\alpha\not\in {\Bbb Z}$  such that $\alpha \Lambda_{\theta}\subseteq \Lambda_{\theta}$.
The above inclusion implies, that $\theta,\alpha\in {\Bbb Q}(\sqrt{D})$ for some integer 
$D>0$.  Indeed, from  $\alpha \Lambda_{\theta}\subseteq \Lambda_{\theta}$  one gets $\alpha=a+b\theta$ and 
$\alpha\theta=c+d\theta$ for some $a,b,c,d\in {\Bbb Z}$,  such that $ad-bc\ne 0$.  
Eliminating $\alpha$ in the above two equations, one obtains a quadratic equation
$b\theta^2+(a-d)\theta-c=0$  with discriminant $D$.  Because $\theta$ is a real root,
one gets $D>0$.   Since   $\alpha=a+b\theta$,  one has  an inclusion $\alpha\in {\Bbb Q}(\sqrt{D})$;
whence the name.

\section{Bon appetit!}
The reader noticed already something unusual:  transformations 
(\ref{iso}) and (\ref{mor}) are given by the same formulas!
 This observation hints at  a possibility of  equivalence of
the corresponding categories as shown  on the diagram below.

\medskip
\begin{picture}(300,110)(-70,-5)
\put(20,70){\vector(0,-1){35}}
\put(130,70){\vector(0,-1){35}}
\put(45,23){\vector(1,0){60}}
\put(45,83){\vector(1,0){60}}
\put(13,20){${\cal A}_{\theta}$}
\put(-10,50){$F$}
\put(150,50){$F$}
\put(123,20){${\cal A}_{\theta'={a\theta+b\over c\theta+d}}$}
\put(17,80){${\cal E}_{\tau}$}
\put(122,80){${\cal E}_{\tau'={a\tau+b\over c\tau+d}}$}

\end{picture}
\begin{example}\label{exm1}
If  $D>1$ is a square-free integer, then  $F$ maps ${\cal E}_{\tau}$ with CM by $\sqrt{-D}$ to 
 ${\cal A}_{\theta}$ with RM  by  $\sqrt{D}$;  in other words,  $F({\cal E}_{\sqrt{-D}})= {\cal A}_{\sqrt{D}}$.
Note however, that no explicit  formula for the  function  $\theta=\theta(\tau)$ exists in general.  
\end{example}

The  {\it category  of elliptic curves}  consists of all ${\cal E}_{\tau}$
and all algebraic (or holomorphic)  maps between ${\cal E}_{\tau}$.
Likewise, the  {\it category of non-commutative tori}  consists of  all
${\cal A}_{\theta}$ and all $C^*$-algebra homomorphisms between 
${\cal A}_{\theta}\otimes {\cal K}$,  where ${\cal K}$
is the $C^*$-algebra of all compact operators on an infinite dimensional separable Hilbert space.  
(It was proved by Marc Aristide Rieffel,  that ${\cal A}_{\theta}\otimes {\cal K}$
and ${\cal A}_{\theta'}\otimes {\cal K}$ are isomorphic  if and only if ${\cal A}_{\theta}$
and ${\cal A}_{\theta'}$ are {\it Morita equivalent}.)
The following result tells us, that indeed we have an equivalence 
of two categories. 
\begin{theorem}\label{thm0}
There exists a covariant functor,  $F$,  from  the category of elliptic curves
to the category of   noncommutative tori,  such that $F({\cal E}_{\tau})$ 
and  $F({\cal E}_{\tau'})$ are isomorphic  if and only if 
 ${\cal E}_{\tau}$ and  ${\cal E}_{\tau'}$ are related by an algebraic  map.
\end{theorem}
\begin{proof}
A  {\it Sklyanin algebra}   $S(\alpha,\beta,\gamma)$  
is  a  ${\Bbb C}$-algebra   on   four  generators  
$x_1,\dots,x_4$   and  six  quadratic relations: 
\begin{equation}\label{sklyanin}
\left\{
\begin{array}{ccc}
x_1x_2-x_2x_1 &=& \alpha(x_3x_4+x_4x_3),\\
x_1x_2+x_2x_1 &=& x_3x_4-x_4x_3,\\
x_1x_3-x_3x_1 &=& \beta(x_4x_2+x_2x_4),\\
x_1x_3+x_3x_1 &=& x_4x_2-x_2x_4,\\
x_1x_4-x_4x_1 &=& \gamma(x_2x_3+x_3x_2),\\ 
x_1x_4+x_4x_1 &=& x_2x_3-x_3x_2,
\end{array}
\right.
\end{equation}
where $\alpha+\beta+\gamma+\alpha\beta\gamma=0$, see \cite{SmiSta1}, p. 260. 
It was proved that the  center of  $S(\alpha,\beta,\gamma)$ gives rise to a family of  elliptic curves  ${\cal E}_{\tau}$ 
 in  the  Jacobi form:
\begin{equation}
\left\{
\begin{array}{ccc}
u^2+v^2+w^2+z^2 &=&  0,\\
{1-\alpha\over 1+\beta}v^2+{1+\alpha\over 1-\gamma}w^2+z^2  &=&  0, 
\end{array}
\right.
\end{equation}
and  an automorphism $\sigma: {\cal E}_{\tau}\to {\cal E}_{\tau}$  \cite{SmiSta1}. 
As explained in Section 1,   $\Coh~({\cal E}_{\tau})\cong \Mod~(S(\alpha,\beta,\gamma))$
for each admissible values of $\alpha, \beta$ and $\gamma$ 
even though $S(\alpha,\beta,\gamma)$   is a non-commutative ring.  
\begin{remark}
The algebra $S(\alpha,\beta,\gamma)$  depends on  two variables, say,  $\alpha$
and $\beta$.  It is not hard to see that  $\alpha$ corresponds  to  $\tau$ and $\beta$ corresponds to   
$\sigma$, since the automorphism  $\sigma$ is given by the shift $0\mapsto p$, where $p$ is a point of 
${\cal E}_{\tau}$.
\end{remark}

Let  ${\cal E}_{\tau}$ be {\it any} elliptic curve,  but choose 
$\sigma: {\cal E}_{\tau}\to {\cal E}_{\tau}$ to be  of {\it order 4}.
Following  (\cite{FeOd1}, Remark 1), in this case  system (\ref{sklyanin}) can be brought to 
the {\it skew-symmetric}  form:
\begin{equation}\label{rel1}
\left\{
\begin{array}{cc}
x_3x_1 &= \mu e^{2\pi i\theta}x_1x_3,\\
x_4x_2 &= {1\over \mu} e^{2\pi i\theta}x_2x_4,\\
x_4x_1 &= \mu e^{-2\pi i\theta}x_1x_4,\\
x_3x_2 &= {1\over \mu} e^{-2\pi i\theta}x_2x_3,\\
x_2x_1 &= x_1x_2,\\
x_4x_3 &= x_3x_4,
\end{array}
\right.
\end{equation}
where $\theta=Arg~(q)$ and $\mu=|q|$ for some complex number $q\in {\Bbb C} - \{0\}$.
(In general,  any Sklyanin algebra on $n$ generators can be brought to the skew-symmetric
form  if and only if  $\sigma^n=Id$ \cite{FeOd1}, Remark 1.)  
We shall denote by $S(q)$ the Sklyanin algebra given by relations (\ref{rel1}).

The reader can verify  that relations (\ref{rel1}) are invariant under the  involution $x_1^*=x_2$ 
and $x_3^*=x_4$. This involution turns $S(q)$ into an algebra with involution.

\medskip
To continue our proof,  we need the following result.
\begin{lemma}\label{lm1}
The system of  relations (\ref{easy}) for any  free  ${\Bbb C}$-algebra
on generators $x_1,\dots,x_4$  is equivalent to  the following system  of   relations:
\begin{equation}\label{rel2}
\left\{
\begin{array}{cc}
x_3x_1 &=  e^{2\pi i\theta}x_1x_3,\\
x_4x_2 &=  e^{2\pi i\theta}x_2x_4,\\
x_4x_1 &=  e^{-2\pi i\theta}x_1x_4,\\
x_3x_2 &=   e^{-2\pi i\theta}x_2x_3,\\
x_2x_1 &= x_1x_2=e,\\
x_4x_3 &= x_3x_4=e.
\end{array}
\right.
\end{equation}
\end{lemma}
\begin{proof}
Notice that  the first  and the two last equations  of (\ref{easy})
and (\ref{rel2}) coincide;  we shall proceed stepwise for the rest  of the equations.

Let us prove that relations (\ref{easy}) imply  $x_1x_4=e^{2\pi i\theta}x_4x_1$. 
It follows from $x_1x_2=e$ and $x_3x_4=e$ that $x_1x_2x_3x_4=e$.  Since $x_1x_2=x_2x_1$ we can bring  the last 
equation to  the form  $x_2x_1x_3x_4=e$ and multiply  both sides by   $e^{2\pi i\theta}$;
thus one gets  $x_2(e^{2\pi i\theta}x_1x_3)x_4=e^{2\pi i\theta}$.  
But $e^{2\pi i\theta}x_1x_3=x_3x_1$ and our main equation takes the form $x_2x_3x_1x_4= e^{2\pi i\theta}$. 
Multiplying  both sides on the left  by  $x_1$  we have   $x_1x_2x_3x_1x_4= e^{2\pi i\theta}x_1$;  since $x_1x_2=e$ 
one  has   $x_3x_1x_4= e^{2\pi i\theta}x_1$.  
Again,  one can multiply both sides  on the left   by  
$x_4$ and thus get  $x_4x_3x_1x_4= e^{2\pi i\theta}x_4x_1$; since $x_4x_3=e$,
one gets the required identity  $x_1x_4= e^{2\pi i\theta}x_4x_1$.

The proof of the remaining relations $x_4x_2 =  e^{2\pi i\theta}x_2x_4$
and $x_3x_2 =   e^{-2\pi i\theta}x_2x_3$ is similar and is  left to the reader.  
Lemma \ref{lm1} is proved.
\end{proof}

\bigskip
Returning  to the proof of Theorem \ref{thm0},   we see that  relations (\ref{rel1}) are  equivalent to  (\ref{rel2}) plus 
the following  relations: 
\begin{equation}\label{rel0}
x_1x_2=x_3x_4={1\over\mu}e.
\end{equation}
(The reader is encouraged to verify the equivalence.)   We shall call (\ref{rel0})   the {\it scaling of the unit}  relations
and   denote by  $I_{\mu}$  the two-sided ideal in $S(q)$ generated by such  relations. 
It is easy to see  that  the scaling of the unit relations  are invariant under the  involution $x_1^*=x_2$ 
and $x_3^*=x_4$.   Thus   involution on  algebra $S(q)$  extends  to an involution on the 
quotient algebra   $S(q) / I_{\mu}$.  
In view of the above,  one gets  an isomorphism of  rings  with involution: 
\begin{equation}\label{rings}
{\cal A}_{\theta}\cong S(q) / I_{\mu}.
\end{equation}

\bigskip
Theorem \ref{thm0} is proved.   
\end{proof}

\section{Know your rank}
Why is Theorem \ref{thm0}  useful?  Roughly, it says that instead of the elliptic curves  ${\cal E}_{\tau}$,  one 
can study the corresponding non-commutative tori ${\cal A}_{\theta}$;  in particular,  geometry of 
${\cal E}_{\tau}$  can be recast  in  terms of  ${\cal A}_{\theta}$.   
It is interesting, for instance, to compare isomorphism   invariants of ${\cal E}_{\tau}$ with  invariants 
of the algebra ${\cal A}_{\theta}$.  Below we relate one such invariant,  the rank of ${\cal E}_{\tau}$,
to the so-called  arithmetic complexity of ${\cal A}_{\theta}$.

Recall that if $a, b\in {\cal E}_{\tau}$ are two points, then their sum  is
defined as shown in Figure 1.  (The addition  is simpler on the complex torus 
${\Bbb  C}/L_{\tau}$, where  the sum is just the image of sum of complex numbers  $a,b\in {\Bbb C}$
under projection  ${\Bbb C}\to {\Bbb  C}/L_{\tau}$.) 
Since each point has an additive inverse and there is a zero point, 
the set of all points of ${\cal E}_{\tau}$ has the structure of an {\it abelian group}. 
If ${\cal E}_{\tau}$ is defined over ${\Bbb Q}$ (or a finite extension of ${\Bbb Q}$),
then the Mordell-Weil theorem says that the  subgroup of points with rational coordinates
is  a  {\it finitely generated} abelian group.  The number of generators in this group  is called the {\it rank} 
of ${\cal E}_{\tau}$ and is denoted by $rk({\cal E}_{\tau})$. Little is known about  ranks
in general, except for the famous {\it Birch and Swinnerton-Dyer Conjecture} comparing 
$rk({\cal E}_{\tau})$  to the order at zero of the Hasse-Weil  $L$-function attached to ${\cal E}_{\tau}$.

Let  ${\cal E}_{\tau}$ be an elliptic curve with CM by $\sqrt{-D}$
for some square-free positive integer $D$.
In this case $rk({\cal E}_{\tau})$ is finite,  because ${\cal E}_{\tau}$
can be defined over ${\Bbb Q}$ or a finite extension of ${\Bbb Q}$;
moreover,  the integer $rk({\cal E}_{\tau})$
is an isomorphism  invariant of ${\cal E}_{\tau}$  modulo a  
finite number of twists.  
The ranks of  CM elliptic curves were calculated  by Benedict Hyman Gross (born 1950)
in his doctoral thesis  \cite{G};  they are reproduced in the table below for some primes  
$D\equiv 3 ~mod~ 4$.

Let ${\cal A}_{\theta}$ be a non-commutative
torus with RM by $\sqrt{D}$, see Example \ref{exm1} for motivation. 
By an {\it arithmetic complexity} $c({\cal A}_{\theta})$ of ${\cal A}_{\theta}$
one understands  the number of independent entries 
in the period $(a_1,\dots,a_n)$ of continued fraction of $\sqrt{D}$.
Roughly speaking, the $c({\cal A}_{\theta})$ is equal to  the Krull dimension of an irreducible 
component of the affine variety defined by a diophantine equation obtained from the continued 
fraction of   $\sqrt{D}$;
we refer the interested reader to  \cite[Section 6.2.1]{Nik2} for the details. 
It follows from the standard properties of continued fractions and 
formula (\ref{mor}), that $c({\cal A}_{\theta})$  is 
an isomorphism invariant of the  algebra  ${\cal A}_{\theta}$.
The values of $c({\cal A}_{\theta})$  for primes  
$D\equiv 3 ~mod~ 4$ are  given  in the table below.

\medskip
\begin{tabular}{c|c|c|c}
\hline
&&continued&\\
$D\equiv 3 ~mod ~4$ & $rk ({\cal E}_{\tau})$ & fraction of $\sqrt{D}$ & $c({\cal A}_{\theta})$\\
&&&\\
\hline
$3$ & $1$ & $[1,\overline{1,2}]$ & $2$\\
\hline
$7$ & $0$ & $[2,\overline{1,1,1,4}]$ & $1$\\
\hline
$11$ & $1$ & $[3,\overline{3,6}]$ & $2$\\
\hline
$19$ & $1$ & $[4,\overline{2,1,3,1,2,8}]$ & $2$\\
\hline
$23$ & $0$ & $[4,\overline{1,3,1,8}]$ & $1$\\
\hline
$31$ & $0$ & $[5,\overline{1,1,3,5,3,1,1,10}]$ & $1$\\
\hline
$43$ & $1$ & $[6,\overline{1,1,3,1,5,1,3,1,1,12}]$ & $2$\\
\hline
$47$ & $0$ & $[6,\overline{1,5,1,12}]$ & $1$\\
\hline
$59$ & $1$ & $[7,\overline{1,2,7,2,1,14}]$ & $2$\\
\hline
$67$ & $1$ & $[8,\overline{5,2,1,1,7,1,1,2,5,16}]$ & $2$\\
\hline
$71$ & $0$ & $[8,\overline{2,2,1,7,1,2,2,16}]$ & $1$\\
\hline
$79$ & $0$ & $[8,\overline{1,7,1,16}]$ & $1$\\
\hline
$83$ & $1$ & $[9,\overline{9,18}]$ & $2$\\
\hline
\end{tabular}

\vskip1cm
The reader may notice a correlation  between $rk({\cal E}_{\tau})$ and $c({\cal A}_{\theta})$, 
which can be generalized as follows.   
\begin{theorem}\label{thm00}
{\bf (\cite{Nik2})}
If $D\equiv 3 ~mod ~4$ is a prime number, then
\begin{equation}\label{reconcile}
rk({\cal E}_{\tau})=c({\cal A}_{\theta})-1.
\end{equation}
\end{theorem}

Perhaps  (\ref{reconcile}) is  a reconciliation formula for the 
``opposites'' ${\cal E}_{\tau}$ and  ${\cal A}_{\theta}$   meant  by 
the  science fiction writer Walter M. Miller Jr. (1923-1996).  The reader is to judge!

\bigskip
\noindent \textbf{Acknowledgements.} 
I thank Sasha ~Martsinkovsky for an interest, many useful comments and an invitation 
to talk at his  Representation Theory Seminar at the Northeastern University in Boston.  
I am grateful to the referee for careful reading of the manuscript.

{\footnotesize Department of Mathematics and Computer Science}

{\footnotesize St.~John's University, 8000 Utopia Parkway,}

{\footnotesize New York,  NY 11439, United States.}

{\footnotesize e-mail: igor.v.nikolaev@gmail.com}




\end{document}